\documentclass[11pt]{amsart}
\usepackage[centertags]{amsmath}
\usepackage{amsfonts}
\usepackage{amssymb}
\usepackage{amsthm}
\usepackage{newlfont}
\usepackage{amscd}
\usepackage{mathrsfs}
\usepackage[all,cmtip]{xy}
\usepackage{tikz-cd}
\tikzcdset{arrow style=Latin Modern, row sep/normal=1.35em, column sep/normal=1.8em, cramped}
\usepackage{tikz}
\usepackage[pagebackref=false,colorlinks]{hyperref}

\definecolor{mycolor}{HTML}{F7F8E0}
\definecolor{myorange}{RGB}{245,156,74}
\definecolor{cadetgrey}{rgb}{0.57, 0.64, 0.69}
\definecolor{calpolypomonagreen}{rgb}{0.12, 0.3, 0.17}

\hypersetup{pdffitwindow=true,linkcolor=calpolypomonagreen,citecolor=calpolypomonagreen,urlcolor=calpolypomonagreen}
\usepackage{cite}
\usepackage{fancyhdr}
\usepackage{stmaryrd}

\usepackage[OT2,OT1]{fontenc}
\newcommand\cyr{%
\renewcommand\rmdefault{wncyr}%
\renewcommand\sfdefault{wncyss}%
\renewcommand\encodingdefault{OT2}%
\normalfont
\selectfont}
\DeclareTextFontCommand{\textcyr}{\cyr}

\usepackage[toc, page] {appendix}

\numberwithin{equation}{section}

\newtheorem{thm}{Theorem}[section]

\newtheorem{cor}[thm]{Corollary}
\newtheorem{lem}[thm]{Lemma}
\newtheorem{prop}[thm]{Proposition}

\newtheorem{conj}[thm]{Conjecture}

\theoremstyle{definition}

\newtheorem{rem}[thm]{Remark}

\newcommand{\sha}{\textrm{{\cyr SH}}}

\begin{document}
\title[Soft $p$-converse to GZK]{On the soft $p$-converse to a theorem of Gross--Zagier and Kolyvagin}
\author{Chan-Ho Kim}
\address{Center for Mathematical Challenges, Korea Institute for Advanced Study, 85 Hoegiro, Dongdaemun-gu, Seoul 02455, Republic of Korea}
\email{chanho.math@gmail.com}
\date{\today}
\begin{abstract}
We give a proof of a soft version of the $p$-converse to a theorem of Gross--Zagier and Kolyvagin for non-CM elliptic curves with good ordinary reduction at $p >3$ under the irreducibility assumption on the residual representation.
In particular, no condition on the conductor is imposed. Combining with the known results, we obtain the equivalence
$$ \mathrm{rk}_{\mathbb{Z}}E(\mathbb{Q}) = 1, \#\sha(E/\mathbb{Q}) < \infty \Leftrightarrow \mathrm{ord}_{s=1}L(E, s) =1$$
for \emph{every} elliptic curve $E$ over $\mathbb{Q}$.
\end{abstract}
\thanks{Data Availability Statements: Data sharing not applicable to this article as no datasets were generated or analysed during the current study.}
\thanks{Statements and Declarations: The author has no relevant financial or non-financial interests to disclose.}

\maketitle

\setcounter{tocdepth}{1}

\section{Statement of the main result}
Let $E$ be an elliptic curve over $\mathbb{Q}$ and $p$ be a prime. The $p^\infty$-Selmer group of $E$ is defined by
$$\mathrm{Sel}(\mathbb{Q}, E[p^\infty])  =\mathrm{ker} \left( \mathrm{H}^1(\mathbb{Q}, E[p^\infty])  \to \prod_v \frac{\mathrm{H}^1(\mathbb{Q}_v, E[p^\infty])}{E(\mathbb{Q}_v) \otimes \mathbb{Q}_p/\mathbb{Z}_p} \right)$$
where $v$ runs over the places of $\mathbb{Q}$, and it encodes the arithmetic of $E$ via the fundamental exact sequence
\[
\xymatrix{
0 \ar[r] & E(\mathbb{Q}) \otimes \mathbb{Q}_p/\mathbb{Z}_p \ar[r] & \mathrm{Sel}(\mathbb{Q}, E[p^\infty]) \ar[r] & \sha(E/\mathbb{Q})[p^\infty] \ar[r] & 0.
}
\]
Denote by $V$ the two dimensional Galois representation over $\mathbb{Q}_p$ associated to $E$.
The Selmer group of $V$ is defined by
$\mathrm{Sel}(\mathbb{Q}, V)  =\mathrm{ker} \left( \mathrm{H}^1(\mathbb{Q}, V)  \to \prod_v \frac{\mathrm{H}^1(\mathbb{Q}_v, V)}{E(\mathbb{Q}_v) \otimes \mathbb{Q}_p} \right)$ in the same manner.

The aim of this article is to give a succinct proof of the following \emph{$p$-converse} result to the theorem of Gross--Zagier and Kolyvagin on elliptic curves of analytic rank one.
\begin{thm} \label{thm:main}
If
\begin{enumerate}
\item[(cork1)] $\mathrm{cork}_{\mathbb{Z}_p} \mathrm{Sel}(\mathbb{Q}, E[p^\infty]) = 1$,
\item[(PRC)] Perrin-Riou's conjecture on Kato's zeta elements holds (Conjecture \ref{conj:perrin-riou}),
\item[(IMC)] the Iwasawa main conjecture (inverting $p$) holds (Conjecture \ref{conj:imc-inverting-p}), and
\item[(res)] the restriction map $\mathrm{res}_p : \mathrm{Sel}(\mathbb{Q}, V) \to E(\mathbb{Q}_p) \otimes \mathbb{Q}_p$ is an isomorphism,
\end{enumerate}
then $\mathrm{ord}_{s=1}L(E, s) =1$. In particular, $\mathrm{rk}_{\mathbb{Z}}E(\mathbb{Q}) = 1$ and $\#\sha(E/\mathbb{Q}) < \infty$.
\end{thm}
Perrin-Riou's conjecture is recently proved by Bertolini--Darmon--Venerucci \cite{bertolini-darmon-venerucci} and Burungale--Skinner--Tian \cite{burungale-skinner-tian} independently (Theorem \ref{thm:bertolini-darmon-venerucci}).
Also, the Iwasawa main conjecture (inverting $p$) is proved for non-CM elliptic curves with good ordinary reduction at $p >3$ by Kato \cite{kato-euler-systems}, Skinner--Urban \cite{skinner-urban}, and Wan \cite{wan_hilbert} if $E[p]$ is an irreducible Galois representation (Theorem \ref{thm:imc-inverting-p}).
As an application of these results, we obtain the following statement.
\begin{cor}\label{cor:good_ordinary}
Let $E$ be a non-CM elliptic curve over $\mathbb{Q}$ and $p>3$ a good ordinary prime for $E$ such that $E[p]$ is an irreducible mod $p$ Galois representation.
If
\begin{enumerate}
\item[(cork1)] $\mathrm{cork}_{\mathbb{Z}_p} \mathrm{Sel}(\mathbb{Q}, E[p^\infty]) = 1$, and
\item[(res)] the restriction map $\mathrm{res}_p : \mathrm{Sel}(\mathbb{Q}, V) \to E(\mathbb{Q}_p) \otimes \mathbb{Q}_p$ is an isomorphism,
\end{enumerate}
then $\mathrm{ord}_{s=1}L(E, s) =1$.
In particular, $\mathrm{rk}_{\mathbb{Z}}E(\mathbb{Q}) = 1$ and $\#\sha(E/\mathbb{Q}) < \infty$.
\end{cor}
\begin{rem}
\begin{enumerate}
\item Assumption (res) also appears in \cite[Theorem B]{skinner-converse} in a slightly different context. See also \cite[Remark 2.9.1.(x)]{skinner-converse}. 
The $p$-converse \emph{with} Assumption (res) is called the \emph{soft} $p$-converse. 
\item When $p=3$, Corollary \ref{cor:good_ordinary} still holds if we further assume that there exists a prime $\ell$ exactly dividing the conductor of $E$ such that $E[p]$ is ramified at $\ell$ \cite[Theorem 3.33]{skinner-urban}.
\item For a non-CM elliptic curve $E$, the density of good ordinary primes for $E$ is one \cite[Corollaire 2 to Th\'{e}or\`{e}me 20, $\S$8.2]{serre-density} and $E[p]$ is an irreducible Galois representation for a sufficiently large prime $p \gg 0$ \cite[Theorem in $\S$2.1 of Chapter IV]{serre-abelian-l-adic}.
\end{enumerate}
\end{rem}
Since Assumption (res) holds if we further assume $\#\sha(E/\mathbb{Q})[p^\infty] < \infty$ in Corollary \ref{cor:good_ordinary}, we obtain the following statement.
\begin{cor} \label{cor:converse}
Let $E$ be a non-CM elliptic curve over $\mathbb{Q}$.
If $\mathrm{rk}_{\mathbb{Z}}E(\mathbb{Q}) = 1$ and $\#\sha(E/\mathbb{Q})[p^\infty] < \infty$ for at least one good ordinary prime $p > 3$ such that $E[p]$ is irreducible, then $\mathrm{ord}_{s=1}L(E, s) =1$. In particular, $\#\sha(E/\mathbb{Q}) < \infty$.
\end{cor}
\begin{rem}
We are informed that Perrin-Riou's conjecture and Corollaries \ref{cor:good_ordinary} and \ref{cor:converse} also appear in \cite{burungale-skinner-tian}.
\end{rem}
In their pioneering works, Skinner \cite{skinner-converse} and W. Zhang \cite{wei-zhang-mazur-tate} independently proved the first general results towards the $p$-converse under certain assumptions on the conductor of elliptic curves.
 In \cite[Theorem A]{skinner-converse}, the conductor is square-free and satisfies a certain existence condition on split/non-split reduction primes.
 In \cite[Theorem 1.3]{wei-zhang-mazur-tate}, the conductor satisfies a certain ramification condition arising from the arithmetic of Shimura curves.
Our result completely removes these assumptions on the conductor.

Together with the work of Rubin \cite{rubin-p-converse} and Burungale--Tian \cite{burungale-tian-p-converse} on the CM case,
Corollary \ref{cor:converse} completes the converse to the theorem of Gross--Zagier and Kolyvagin for elliptic curves over $\mathbb{Q}$. In other words, we have the equivalence
$$\mathrm{rk}_{\mathbb{Z}}E(\mathbb{Q}) = 1, \#\sha(E/\mathbb{Q}) < \infty \Leftrightarrow \mathrm{ord}_{s=1}L(E, s) =1$$
for \emph{every} elliptic curve $E$ over $\mathbb{Q}$.

On the other hand, the $p$-converse \emph{without} Assumption (res), which is called the \emph{strong} $p$-converse, can be obtained from the Heegner point main conjecture and other standard ingredients as explained in \cite{wan-heegner, burungale-castella-kim, castella-wan-perrin-riou-ss}. The Heegner point main conjecture can be replaced by the Kolyvagin conjecture  \cite{wei-zhang-mazur-tate}.
We avoid using the Heegner point main conjecture here since it is currently not proved in general despite significant progresses towards it.
Even the formulation of the conjecture at additive reduction primes is an open problem.

The $p$-converse itself is studied extensively in various forms \cite{venerucci-converse, burungale-tian-p-converse, wan-heegner, castella-grossi-lee-skinner, burungale-castella-skinner-tian, skinner-zhang, castella-wan-perrin-riou-ss, sweeting-kolyvagin, burungale-skinner-tian}.
See also \cite{burungale-skinner-tian-survey} for the survey of the recent developments.

\section{Preliminaries}
\subsection{Kato's zeta elements}
Let $T$ be the Tate module of an elliptic curve $E$ and $V = T \otimes \mathbb{Q}_p$.
Denote by $\omega_E$ the N\'{e}ron differential of a global minimal Weierstrass equation for $E$, and
$\Omega^\pm_E$ the real and imaginary N\'{e}ron periods for $E$, respectively.
Adapting the convention of \cite[$\S$6.3 and Theorem 16.2]{kato-euler-systems}, we have the period map $\mathrm{per}_E$ from the space of global differential one-forms on $E$ to the first Betti cohomology of $E$ over $\mathbb{C}$
such that
$$\mathrm{per}_E(\omega_E) = \Omega^+_E \cdot \gamma^+ + \Omega^-_E \cdot \gamma^-$$
for some non-zero $\gamma^{\pm} \in V_{\mathbb{Q}}(-1)^{c = \pm 1}$, respectively, where
$V_{\mathbb{Q}}(-1)$ is the first Betti cohomology of $E$ over $\mathbb{Q}$ and $c$ is the complex conjugation.
We identify $V_{\mathbb{Q}}(-1) \otimes \mathbb{Q}_p$ with the first $p$-adic \'{e}tale cohomology $V(-1)$ of $E$ over $\mathbb{Q}_p$, so $\gamma^{\pm} \in V(-1)$.

Let $\gamma = \gamma^+ + \gamma^- \in V(-1)$.
Following \cite[Theorem 12.5.(1)]{kato-euler-systems}, Kato's zeta element 
$$\mathbf{z}^{(p)}_\gamma \in \varprojlim_n \mathrm{H}^1(\mathbb{Q}(\zeta_{p^n}), V(-1)) $$
for $V(-1)$ over $\mathbb{Q}(\zeta_{p^\infty})$ is obtained from $\gamma$ where $\varprojlim_n$ is taken with respect to the corestriction.
Kato's zeta element $z_{\mathrm{Kato}} \in \mathrm{H}^1(\mathbb{Q}, V)$ for $V$ over $\mathbb{Q}$ is defined by the image of $\mathbf{z}^{(p)}_\gamma$ under the composition
\[
\xymatrix{
\varprojlim_n \mathrm{H}^1(\mathbb{Q}(\zeta_{p^n}), V(-1)) \ar[rr]^-{\otimes (\zeta_{p^n})_{n\geq1}} & &
\varprojlim_n \mathrm{H}^1(\mathbb{Q}(\zeta_{p^n}), V)
\ar[r] &  \mathrm{H}^1(\mathbb{Q}, V) .
}
\]
Let $\mathbb{Q}_\infty$ be the cyclotomic $\mathbb{Z}_p$-extension of $\mathbb{Q}$ and $\Lambda = \mathbb{Z}_p\llbracket \mathrm{Gal}(\mathbb{Q}_{\infty}/\mathbb{Q}) \rrbracket$ the Iwasawa algebra.
Write $\mathrm{H}^1_{\mathrm{Iw}}(\mathbb{Q},V) = \varprojlim_n \mathrm{H}^1(\mathbb{Q}_n,V)$
where $\mathbb{Q}_n$ is the cyclic subextension of $\mathbb{Q}$ of order $p^n$ in $\mathbb{Q}_\infty$.
Kato's zeta element  $z^{\infty}_{\mathrm{Kato}} \in \mathrm{H}^1_{\mathrm{Iw}}(\mathbb{Q},V)$ for $V$ over $\mathbb{Q}_\infty$ is similarly defined by the image of 
$\mathbf{z}^{(p)}_\gamma$ in $\mathrm{H}^1_{\mathrm{Iw}}(\mathbb{Q},V)$.
\begin{rem}
If $p$ is odd and the image of the Galois representation $\rho : \mathrm{Gal}(\overline{\mathbb{Q}}/\mathbb{Q}) \to \mathrm{GL}(V)$
contains a conjugate of $\mathrm{SL}_2(\mathbb{Z}_p)$, then $\gamma$ can be chosen in $T(-1)$ and all the zeta elements lie in the Galois cohomologies with integral coefficients. In other words, $V$ can be replaced by $T$ in the above discussion. See \cite[Theorem 5.1]{kurihara-invent} and  \cite[Theorem 12.5.(4)]{kato-euler-systems}.
\end{rem}
Kato's explicit reciprocity law \cite[Theorem 12.5.(1)]{kato-euler-systems} says that
\begin{equation} \label{eqn:explicit-reciprocity-law}
\mathrm{exp}^* \circ \mathrm{res}_p ( z_{\mathrm{Kato}} ) = (1 - a_p(E) \cdot p^{-1} + p^{-1} ) \cdot \dfrac{L(E, 1)}{\Omega^+_E} \cdot \omega_E
\end{equation}
where
 $\mathrm{res}_p : \mathrm{H}^1(\mathbb{Q}, V) \to \mathrm{H}^1(\mathbb{Q}_p, V)$ is the restriction map and
 $\mathrm{exp}^*$ is Bloch--Kato's dual exponential map.
\subsection{Perrin-Riou's conjecture}
We recall Perrin-Riou's conjecture on Kato's zeta elements \cite[Conjecture in $\S$3.3.2]{perrin-riou-rational-pts}.
\begin{conj}[Perrin-Riou] \label{conj:perrin-riou}
Let $E$ be an elliptic curve over $\mathbb{Q}$ and $p$ be a prime.
If $L(E,1) = 0$, then there exists a global point $P \in E(\mathbb{Q})$ satisfying the following properties.
\begin{enumerate}
\item The point $P$ has infinite order if and only if $\mathrm{ord}_{s=1}L(E, s) =1$.
\item The following equality holds in $\mathbb{Q}_p$ up to multiplication by a non-zero rational number:
$$\mathrm{log}_{\omega_E}(\mathrm{res}_p (z_{\mathrm{Kato}})) = \left( \mathrm{log}_{\omega_E}(P) \right)^2$$
where $\mathrm{log}_{\omega_E}$ is the $p$-adic Lie group logarithm corresponding to $\omega_E$.
\end{enumerate}
\end{conj}
Conjecture \ref{conj:perrin-riou} is recently proved by Bertolini--Darmon--Venerucci \cite[Theorem A]{bertolini-darmon-venerucci} and Burungale--Skinner--Tian \cite{burungale-skinner-tian} independently.
\begin{thm}[Bertolini--Darmon--Venerucci] \label{thm:bertolini-darmon-venerucci}
If $E$ has semi-stable reduction at an odd prime $p$, then Conjecture \ref{conj:perrin-riou} is true.
\end{thm}
\begin{rem}
In \cite{burungale-skinner-tian}, Theorem \ref{thm:bertolini-darmon-venerucci} is proved for a good reduction prime $p \geq 5$.
\end{rem}
We also refer the reader to \cite{venerucci-perrin-riou, kazim-iwasawa-2017, kazim-pollack-sasaki} for other approaches towards Perrin-Riou's conjecture.

\subsection{Iwasawa main conjecture}
We review the Iwasawa main conjecture (inverting $p$) without $p$-adic $L$-functions \`{a} la Kato \cite[$\S$6]{kurihara-invent}, \cite[Conjecture 12.10]{kato-euler-systems}.

The $p$-strict Selmer group of $E[p^\infty]$ over $\mathbb{Q}_n$ is defined by
$$\mathrm{Sel}_0(\mathbb{Q}_n, E[p^\infty])  = \mathrm{ker} \left( \mathrm{res}_p : \mathrm{Sel}(\mathbb{Q}_n, E[p^\infty]) \to E(\mathbb{Q}_{n,p}) \otimes \mathbb{Q}_p/\mathbb{Z}_p \right)$$
where $\mathbb{Q}_{n,p}$ is the $p$-adic completion of $\mathbb{Q}_{n}$.
Write $\mathrm{Sel}_0(\mathbb{Q}_{\infty}, E[p^\infty]) = \varinjlim_n \mathrm{Sel}_0(\mathbb{Q}_n, E[p^\infty])$ and $(-)^\vee = \mathrm{Hom}_{\mathbb{Z}_p}(-, \mathbb{Q}_p/\mathbb{Z}_p)$.
The rational version of the Iwasawa main conjecture can be written as follows.
\begin{conj}[IMC] \label{conj:imc-inverting-p}
As ideals of $\Lambda \otimes \mathbb{Q}_p$, we have
$$\mathrm{char}_{\Lambda \otimes \mathbb{Q}_p}
\left( \dfrac{\mathrm{H}^1_{\mathrm{Iw}}(\mathbb{Q},V)}{ z^{\infty}_{\mathrm{Kato}}} \right) 
=
\mathrm{char}_{\Lambda \otimes \mathbb{Q}_p} \left( \mathrm{Sel}_0(\mathbb{Q}_{\infty}, E[p^\infty])^\vee \otimes_{\mathbb{Z}_p} \mathbb{Q}_p \right) .
$$
\end{conj}
\begin{thm}[Kato, Skinner--Urban, Wan] \label{thm:imc-inverting-p}
If $E$ has good ordinary reduction at $p \geq 5$ and $E[p]$ is an irreducible Galois representation, then Conjecture \ref{conj:imc-inverting-p} is true.
\end{thm}
\begin{proof}
See \cite[Theorem 12.5 and Theorem 17.4]{kato-euler-systems}, \cite[Theorem 3.33]{skinner-urban}, and \cite[Theorem 4]{wan_hilbert}.
\end{proof}

\section{Proof of Theorem \ref{thm:main}}
We extract the following statement from Conjecture \ref{conj:perrin-riou}.
\begin{cor} \label{cor:bertolini-darmon-venerucci}
If Conjecture \ref{conj:perrin-riou} holds, then the following statements are equivalent.
\begin{enumerate}
\item $\mathrm{ord}_{s=1} L(E, s) = 1$.
\item The global point $P$ in Conjecture \ref{conj:perrin-riou} has infinite order.
\item $\mathrm{res}_p (z_{\mathrm{Kato}})$ is non-zero.
\end{enumerate}
\end{cor}
Although the $L(E,1) = 0$ assumption is incorporated in Conjecture \ref{conj:perrin-riou}, it can be removed in Theorem \ref{thm:main} thanks to the following theorem.
\begin{thm}[Gross--Zagier, Kolyvagin, Rubin, Kato] \label{thm:kato}
Let $E$ be an elliptic curve over $\mathbb{Q}$.
If $L(E,1) \neq 0$, then $\mathrm{Sel}(\mathbb{Q}, E[p^\infty])$ is finite.
\end{thm}
\begin{proof}
See \cite{gross-zagier-original} and \cite{kolyvagin-euler-systems} for the general case. See also \cite[Theorem A]{rubin-tate-shafarevich} for the CM case and \cite[Theorem 8.1]{rubin-es-mec} and \cite[Theorem 14.2]{kato-euler-systems} for the non-CM case.
\end{proof}

\begin{prop} \label{prop:z-kato-sel-0-finite}
Let $E$ be an elliptic curve  over $\mathbb{Q}$ and $p$ be a prime.
If the Iwasawa main conjecture inverting $p$ holds (Conjecture \ref{conj:imc-inverting-p}) and  $\mathrm{cork}_{\mathbb{Z}_p} \mathrm{Sel}(\mathbb{Q}, E[p^\infty]) = 1$, then the following statements are equivalent.
\begin{enumerate}
\item $\mathrm{Sel}_0(\mathbb{Q}, E[p^\infty])$ is finite.
\item $z_{\mathrm{Kato}}$ is non-zero.
\item $\mathrm{res}_p (z_{\mathrm{Kato}})$ is non-zero.
\end{enumerate}
\end{prop}
\begin{proof}
Write $f_0 , f_z \in \Lambda \otimes \mathbb{Q}_p$ to be the distinguished polynomials such that
\[
\xymatrix{
 \left( f_0 \right) = \mathrm{char}_{\Lambda \otimes \mathbb{Q}_p} \left( \mathrm{Sel}_0(\mathbb{Q}_{\infty}, E[p^\infty])^\vee \otimes_{\mathbb{Z}_p} \mathbb{Q}_p \right) , &
\left( f_z \right) = \mathrm{char}_{\Lambda\otimes \mathbb{Q}_p } \left( \dfrac{\mathrm{H}^1_{\mathrm{Iw}}(\mathbb{Q},V)}{ z^{\infty}_{\mathrm{Kato}} } \right) .
}
\]
as ideals of $\Lambda \otimes \mathbb{Q}_p$, respectively.
The control theorem for $p$-strict Selmer groups says that the restriction map
$$\mathrm{res} : \mathrm{Sel}_0(\mathbb{Q}, E[p^\infty]) \to \mathrm{Sel}_0(\mathbb{Q}_{\infty}, E[p^\infty])^\Gamma$$
has finite kernel and cokernel where $\Gamma = \mathrm{Gal}(\mathbb{Q}_\infty/\mathbb{Q})$.
Thus, the finiteness of $\mathrm{Sel}_0(\mathbb{Q}, E[p^\infty])$ is equivalent to the finiteness of $\mathrm{Sel}_0(\mathbb{Q}_{\infty}, E[p^\infty])^\Gamma$.
The latter is also equivalent to  $\mathbf{1}(f_0) \neq 0$  where $\mathbf{1}$ is the trivial character.
On the other hand, 
$z_{\mathrm{Kato}} \neq 0$ if and only if $\mathbf{1}(f_z) \neq 0$
since $\mathrm{H}^1_{\mathrm{Iw}}(\mathbb{Q},V) \simeq \Lambda \otimes \mathbb{Q}_p$ \cite[Theorem 12.4.(2)]{kato-euler-systems}.
Conjecture \ref{conj:imc-inverting-p} implies that $\mathbf{1}(f_0) \neq 0$ if and only if $\mathbf{1}(f_z) \neq 0$. Thus, the equivalence between (1) and (2) follows.

It suffices to check $(2) \Rightarrow (3)$ since $(3) \Rightarrow (2)$ is trivial.
We now assume $z_{\mathrm{Kato}}$ is non-zero.
Consider the exact sequence
\begin{equation} \label{eqn:restriction-map}
\xymatrix{
0 \ar[r] & \mathrm{Sel}_0(\mathbb{Q}, V) \ar[r] & \mathrm{Sel}(\mathbb{Q}, V) \ar[r]^-{\mathrm{res}_p} & E(\mathbb{Q}_p) \otimes \mathbb{Q}_p .
}
\end{equation}
By the equivalence between (1) and (2), $\mathrm{Sel}_0(\mathbb{Q}, E[p^\infty])$ is finite.
Since $E[p^\infty] \simeq V/T$, there exists a natural map $\mathrm{Sel}_0(\mathbb{Q}, V) \to \mathrm{Sel}_0(\mathbb{Q}, E[p^\infty])$
coming from the exact sequence $0 \to T \to V \to V/T \to 0$.
Since the kernel and the cokernel of the map lie in compact $\mathbb{Z}_p$-modules, the finiteness of $\mathrm{Sel}_0(\mathbb{Q}, E[p^\infty])$ is equivalent to $\mathrm{Sel}_0(\mathbb{Q}, V) = 0$.
Applying the same argument to usual Selmer groups, we also have the equivalence between $\mathrm{cork}_{\mathbb{Z}_p} \mathrm{Sel}(\mathbb{Q}, E[p^\infty]) = 1$ and $\mathrm{dim}_{\mathbb{Q}_p} \mathrm{Sel}(\mathbb{Q}, V) = 1$.
Thus, the restriction map $\mathrm{res}_p$ is an isomorphism of one-dimensional $\mathbb{Q}_p$-vector spaces.

By Kato's explicit reciprocity law (\ref{eqn:explicit-reciprocity-law}) and Theorem \ref{thm:kato}, we have $z_{\mathrm{Kato}} \in \mathrm{Sel}(\mathbb{Q}, V)$.
Since $z_{\mathrm{Kato}} \neq 0$ and $\mathrm{res}_p$ is an isomorphism,  $\mathrm{res}_p (z_{\mathrm{Kato}})$ is also non-zero.
\end{proof}
\begin{rem}
Although $(2) \Rightarrow (3)$ immediately follows from Assumption (res), we would like to use Assumption (res) minimally.
\end{rem}
The following lemma completes the proof.
\begin{lem} \label{lem:sel-0-sel}
If $\mathrm{cork}_{\mathbb{Z}_p} \mathrm{Sel}(\mathbb{Q}, E[p^\infty]) = 1$ and $\mathrm{res}_p : \mathrm{Sel}(\mathbb{Q}, V) \simeq E(\mathbb{Q}_p) \otimes \mathbb{Q}_p$, then $\mathrm{Sel}_0(\mathbb{Q}, E[p^\infty])$ is finite.
\end{lem}
\begin{proof}
As in the proof of Proposition \ref{prop:z-kato-sel-0-finite},  $\mathrm{cork}_{\mathbb{Z}_p} \mathrm{Sel}(\mathbb{Q}, E[p^\infty]) = 1$ if and only if $\mathrm{dim}_{\mathbb{Q}_p} \mathrm{Sel}(\mathbb{Q}, V) = 1$.
If $\mathrm{res}_p : \mathrm{Sel}(\mathbb{Q}, V) \simeq E(\mathbb{Q}_p) \otimes \mathbb{Q}_p$, then $\mathrm{Sel}_0(\mathbb{Q}, V) = 0$ due to  (\ref{eqn:restriction-map}). Thus, the finiteness of $\mathrm{Sel}_0(\mathbb{Q}, E[p^\infty])$ follows.
\end{proof}
Theorem \ref{thm:main} follows from the combination of Corollary \ref{cor:bertolini-darmon-venerucci}, Theorem \ref{thm:kato}, Proposition \ref{prop:z-kato-sel-0-finite}, and Lemma \ref{lem:sel-0-sel}.
The last ``In particular" part follows from the work of Gross--Zagier \cite{gross-zagier-original} and Kolyvagin \cite{kolyvagin-euler-systems} with a choice of a suitable imaginary quadratic field.
The existence of such an imaginary quadratic field is ensured by the work of Bump--Friedberg--Hoffstein \cite{bump-friedberg-hoffstein} or Murty--Murty \cite{murty-murty-mean}.

\section*{Acknowledgement}
The discussion with Ashay Burungale and Francesc Castella leads us to improve the main result and the exposition in an earlier version significantly.
We deeply thank to both.
This research was partially supported 
by a KIAS Individual Grant (SP054102) via the Center for Mathematical Challenges at Korea Institute for Advanced Study and
by the National Research Foundation of Korea(NRF) grant funded by the Korea government(MSIT) (No. 2018R1C1B6007009).
We would like to thank the referee for carefully reading our manuscript and for giving constructive and valuable comments.

\bibliographystyle{amsalpha}
\bibliography{library}

\providecommand{\bysame}{\leavevmode\hbox to3em{\hrulefill}\thinspace}
\providecommand{\MR}{\relax\ifhmode\unskip\space\fi MR }
\providecommand{\MRhref}[2]{%
  \href{http://www.ams.org/mathscinet-getitem?mr=#1}{#2}
}
\providecommand{\href}[2]{#2}
\begin{thebibliography}{BCST22}

\bibitem[BCK21]{burungale-castella-kim}
Ashay~A. Burungale, Francesc Castella, and Chan-Ho Kim, \emph{A proof of
  {P}errin-{R}iou's {H}eegner point main conjecture}, Algebra Number Theory
  \textbf{15} (2021), no.~10, 1627--1653.

\bibitem[BCST22]{burungale-castella-skinner-tian}
Ashay~A. Burungale, Francesc Castella, Christopher Skinner, and Ye~Tian,
  \emph{{$p^\infty$}-{S}elmer groups and rational points on {CM} elliptic
  curves}, Ann. Math. Qu\'{e}bec \textbf{46} (2022), no.~2, 324--346.

\bibitem[BDV22]{bertolini-darmon-venerucci}
Massimo Bertolini, Henri Darmon, and Rodolfo Venerucci, \emph{Heegner points
  and {B}eilinson--{K}ato elements: a conjecture of {P}errin-{R}iou}, Adv.
  Math. \textbf{398} (2022), 108172.

\bibitem[BFH90]{bump-friedberg-hoffstein}
Daniel Bump, Solomon Friedberg, and Jeffrey Hoffstein, \emph{Nonvanishing
  theorems for {$L$}-functions of modular forms and their derivatives}, Invent.
  Math. \textbf{102} (1990), no.~3, 543--618.

\bibitem[BPS]{kazim-pollack-sasaki}
K{\^{a}}zim B{\"{u}}y{\"{u}}kboduk, Robert Pollack, and Shu Sasaki,
  \emph{{$p$}-adic {G}ross--{Z}agier formula at critical slope and a conjecture
  of {P}errin-{R}iou}, preprint,
  \href{https://arxiv.org/abs/1811.08216}{arXiv:1811.08216}.

\bibitem[BST]{burungale-skinner-tian}
Ashay~A. Burungale, Christopher Skinner, and Ye~Tian, \emph{Elliptic curves and
  {B}eilinson--{K}ato elements: rank one aspects}, preprint.

\bibitem[BST21]{burungale-skinner-tian-survey}
\bysame, \emph{The {B}irch and {S}winnerton-{D}yer conjecture: a brief survey},
  Nine mathematical challenges--an elucidation (Providence, RI) (A.~Kechris,
  N.~Makarov, D.~Ramakrishnan, and X.~Zhu, eds.), Proc. Sympos. Pure Math.,
  vol. 104, Amer. Math. Soc., 2021, pp.~11--29.

\bibitem[BT20]{burungale-tian-p-converse}
Ashay~A. Burungale and Ye~Tian, \emph{{$p$}-converse to a theorem of
  {G}ross--{Z}agier, {K}olyvagin and {R}ubin}, Invent. Math. \textbf{220}
  (2020), no.~1, 211--253.

\bibitem[B{\"{u}}y20]{kazim-iwasawa-2017}
K{\^{a}}zim B{\"{u}}y{\"{u}}kboduk, \emph{{B}eilinson--{K}ato and
  {B}eilinson--{F}lach elements, {C}oleman--{R}ubin--{S}tark classes, {H}eegner
  points and a conjecture of {P}errin-{R}iou}, Development of Iwasawa Theory --
  the Centennial of K. Iwasawa’s Birth (Tokyo) (Masato Kurihara, Kenichi
  Bannai, Tadashi Ochiai, and Takeshi Tsuji, eds.), Adv. Stud. Pure Math.,
  vol.~86, Mathematical {S}ociety of {J}apan, 2020, pp.~141--193.

\bibitem[CGLS22]{castella-grossi-lee-skinner}
Francesc Castella, Giada Grossi, Jaehoon Lee, and Christopher Skinner, \emph{On
  the anticyclotomic {I}wasawa theory of rational elliptic curves at
  {E}isenstein primes}, Invent. Math. \textbf{227} (2022), 517--580.

\bibitem[CW]{castella-wan-perrin-riou-ss}
Francesc Castella and Xin Wan, \emph{Perrin-{R}iou's main conjecture for
  elliptic curves at supersingular primes}, preprint.

\bibitem[GZ86]{gross-zagier-original}
Benedict Gross and Don Zagier, \emph{Heegner points and derivatives of
  {$L$}-series}, Invent. Math. \textbf{84} (1986), no.~2, 225--320.

\bibitem[Kat04]{kato-euler-systems}
Kazuya Kato, \emph{{$p$}-adic {H}odge theory and values of zeta functions of
  modular forms}, Ast\'{e}risque \textbf{295} (2004), 117--290.

\bibitem[Kol90]{kolyvagin-euler-systems}
Victor Kolyvagin, \emph{Euler systems}, The {G}rothendieck {F}estschrift
  {V}olume {II} (Pierre Cartier, Luc Illusie, Nicholas~M. Katz, Gerard Laumon,
  Yuri Manin, and Kenneth~A. Ribet, eds.), Progr. Math., vol.~87,
  Birkh\"{a}user {B}oston, 1990, pp.~435--483.

\bibitem[Kur02]{kurihara-invent}
Masato Kurihara, \emph{On the {T}ate {S}hafarevich groups over cyclotomic
  fields of an elliptic curve with supersingular reduction {I}}, Invent. Math.
  \textbf{149} (2002), 195--224.

\bibitem[MM91]{murty-murty-mean}
M.~Ram Murty and V.~Kumar Murty, \emph{Mean values of derivatives of modular
  {$L$}-series}, Ann. of Math. (2) \textbf{133} (1991), no.~3, 447--475.

\bibitem[PR93]{perrin-riou-rational-pts}
Bernadette Perrin-Riou, \emph{Fonctions {$L$} {$p$}-adiques d'une courbe
  elliptique et points rationnels}, Ann. Inst. Fourier (Grenoble) \textbf{43}
  (1993), no.~4, 945--995.

\bibitem[Rub87]{rubin-tate-shafarevich}
Karl Rubin, \emph{Tate--{S}hafarevich groups and {$L$}-functions of elliptic
  curves with complex multiplication}, Invent. Math. \textbf{89} (1987),
  527--559.

\bibitem[Rub94]{rubin-p-converse}
\bysame, \emph{{$p$}-adic variants of the {B}irch and {S}winnerton-{D}yer
  conjecture for elliptic curves with complex multiplication}, {$p$}-adic
  monodromy and the {B}irch and {S}winnerton-{D}yer conjecture ({B}oston, {MA},
  1991) (Providence, {RI}) (Barry Mazur and Glenn Stevens, eds.), Contemp.
  Math., vol. 165, American Mathematical Society, 1994, pp.~71--80.

\bibitem[Rub98]{rubin-es-mec}
\bysame, \emph{Euler systems and modular elliptic curves}, Galois
  representations in arithmetic algebraic geometry ({D}urham, 1996) (Anthony
  Scholl and Richard Taylor, eds.), London Math. Soc. Lecture Note Ser., vol.
  254, Cambridge University Press, 1998, pp.~351--367.

\bibitem[Ser81]{serre-density}
Jean-Pierre Serre, \emph{Quelques applications du th{\'{e}}or{\`{e}}me de
  densit{\'{e}} de {C}hebotarev}, Publ. Math. Inst. Hautes \'{E}tudes Sci.
  \textbf{54} (1981), 323--401.

\bibitem[Ser89]{serre-abelian-l-adic}
\bysame, \emph{Abelian {$\ell$}-adic representations and elliptic curves},
  Advanced Book Classics, Addison-Wesley Publishing Company, 1989, With the
  collaboration of Willem Kuyk and John Labute.

\bibitem[Ski20]{skinner-converse}
Christopher Skinner, \emph{A converse to a theorem of {G}ross, {Z}agier, and
  {K}olyvagin}, Ann. of Math. \textbf{191} (2020), no.~2, 329--354.

\bibitem[SU14]{skinner-urban}
Christopher Skinner and Eric Urban, \emph{The {I}wasawa main conjectures for
  {$\mathrm{GL}_2$}}, Invent. Math. \textbf{195} (2014), no.~1, 1--277.

\bibitem[Swe]{sweeting-kolyvagin}
Naomi Sweeting, \emph{Kolyvagin's conjecture and patched {E}uler systems in
  anticyclotomic {I}wasawa theory}, preprint,
  \href{https://arxiv.org/abs/2012.11771}{arXiv:2012.11771}.

\bibitem[SZ]{skinner-zhang}
Christopher Skinner and Wei Zhang, \emph{Indivisibility of {H}eegner points in
  the multiplicative case}, preprint,
  \href{https://arxiv.org/abs/1407.1099}{arXiv:1407.1099}.

\bibitem[Ven16a]{venerucci-perrin-riou}
Rodolfo Venerucci, \emph{Exceptional zero formulae and a conjecture of
  {P}errin-{R}iou}, Invent. Math. \textbf{203} (2016), no.~3, 923--972.

\bibitem[Ven16b]{venerucci-converse}
\bysame, \emph{On the {$p$}-converse of the {K}olyvagin--{G}ross--{Z}agier
  theorem}, Comment. Math. Helv. \textbf{91} (2016), no.~3, 397--444.

\bibitem[Wan15]{wan_hilbert}
Xin Wan, \emph{The {I}wasawa main conjecture for {H}ilbert modular forms},
  Forum Math. Sigma \textbf{3} (2015), e18 (95 pages).

\bibitem[Wan21]{wan-heegner}
\bysame, \emph{Heegner point {K}olyvagin system and {I}wasawa main conjecture},
  Acta Math. Sin. (Engl. Ser.) \textbf{37} (2021), no.~1, 104--120.

\bibitem[Zha14]{wei-zhang-mazur-tate}
Wei Zhang, \emph{Selmer groups and the indivisibility of {H}eegner points},
  Camb. J. Math. \textbf{2} (2014), no.~2, 191--253.

\end{thebibliography}

\end{document}